\documentclass[12pt,amscd]{amsart}
\usepackage[all]{xy}
\usepackage{graphicx}
\usepackage{amsmath,amsxtra,amssymb,latexsym, amscd,amsthm}
\usepackage{indentfirst}
\usepackage[mathscr]{eucal}    

\newtheorem{thm}{Theorem}[section]
\newtheorem{cor}[thm]{Corollary}
\newtheorem{lem}[thm]{Lemma}
\newtheorem{prop}[thm]{Proposition}
\theoremstyle{definition}
\newtheorem{defn}[thm]{Definition}
\newtheorem{exm}[thm]{Example}
\newtheorem{rem}[thm]{Remark}

\numberwithin{equation}{section}

\DeclareMathOperator{\depth}{depth}
\DeclareMathOperator{\reg}{reg}

\DeclareMathOperator{\regst}{reg-stab}
\DeclareMathOperator{\conv}{conv}

\DeclareMathOperator{\gstab}{gstab}

\def\Nset{\mathbb {N}}
\def\Zset{\mathbb {Z}}
\def\Qset{\mathbb {Q}}
\def\Rset{\mathbb {R}}

\def\Inba {\overline{I^n}}

\def\albf {{\boldsymbol{\alpha}}}

\def\bbf {\mathbf b}

\def\ebf {\mathbf e}

\def\xbf {\mathbf x}
\def\Xbf {\mathbf X}
\def\Ybf {\mathbf Y}

\def\ubf {\mathbf u}
\def\vbf {\mathbf v}

\def\mfr {\mathfrak m}
\def\nfr {\mathfrak n}

\def\afr {\mathfrak a}

\def\Pcal{\mathcal P}

\def\Jcal{\mathcal J}

\begin{document}

\title[Maximal generating degrees] {Maximal generating degrees of  powers  of homogeneous ideals}
\author{Le Tuan Hoa }
\address{Institute of Mathematics, VAST, 18 Hoang Quoc Viet, 10307 Hanoi, Viet Nam}
\email{lthoa@math.ac.vn}
\date{}
\dedicatory{Dedicated to Professor Nguyen Tu Cuong on the occasion of his 70th birthday.}
\commby{}
\begin{abstract}  The degree excess function $\epsilon(I;n)$ is the difference between the maximal generating degree $d(I^n)$ of a homogeneous ideal $I$ of a polynomial ring and $p(I)n$, where $p(I)$ is the leading coefficient of the asymptotically linear function  $d(I^n)$. It is shown that any non-increasing numerical function can be realized as a degree excess function, and there is a monomial ideal $I$ whose $\epsilon(I;n)$ has exactly a given number of local maxima. In the case of monomial ideals, an upper bound on $\epsilon(I;n)$ is provided. As an application it is shown that in the worst case, the  so-called stability index of the Castelnuovo-Mumford regularity of a monomial ideal $I$ must be at least an exponential function of the number of variables.
\end{abstract}

\maketitle
\section*{Introduction}

In this paper we study the maximal generating degree $d(I^n)$ of powers of a homogeneous ideal $I$ of a polynomial $R= K[X_1,...,X_s]$ of $s$ variables over a field $K$. It is a trivial fact that this function is bounded by $d(I)n$, but it is not obvious that $d(I^n) = p(I) n + c(I)$ for some $p(I)>0$ and $c(I) \geq 0$, when $n\gg 0$. This fact was established in \cite[Theorem 3.1]{CHT}. It is natural to ask how is the initial behavior of this function and when does this function become stable?  Some recent results are related to this kind of problems. A comparison on the numbers of generators $\mu(I)$ and $\mu(I^2)$ of $I$ and $I^2$, respectively, in \cite{HSZ} leads to the study of the initial behavior of the function $\mu(I^n)$. A somewhat surprising result says that given a number $q$ there is a monomial ideal $I$ with the property that the function which describes the number of generators of $I^n$ has at least $q$ local maxima, see \cite{AHZ}.  It is even more surprising, that a conjecture of Herzog and Hibi is recently proved in \cite{HNTT}, which states that the function $\depth (R/I^n)$  can be any convergent numerical function. The study of the maximal generating degree function $d(\Inba)$ of integral closures of a monomial ideal $I$ is carried out in \cite{H3}. The readers  may consult the survey paper \cite{H1} for other related problems and an extended list of references.

 Back to the maximal generating degrees, since  $d(I^n) \geq p(I) n$ by \cite[Proposition 4]{Ko}, it is equivalent to pose the above questions on the function $\epsilon(I; n) = d(I^n) - p(I) n$, which is called the degree excess function of $I$ in this paper.  Due to some obvious restrictions (see Lemma \ref{G1}), not any convergent numerical function can be realized as a degree excess function. For an example, when $n$ is bigger than the so-called reduction number $r(I)$ (see (\ref{RedD}) for the definition), this function must be  non-increasing. However, in this paper, we show that in general degree excess functions could be quite complicate. After showing that any non-increasing numerical function can be realized as a degree excess function (Theorem \ref{G6}), we show that there is an ideal that has exactly a given number of local maxima (Theorem \ref{G7}). Moreover these ideals can be taken to be monomial.
 
In the case of monomial ideals we also give an effective way to compute the number $p(I)$ mentioned above. It is the maximal length of vertices of the Newton polyhedron associated to $I$ (see Theorem \ref{Lead}(i)). This is achieved by using a technique from Linear Programming. Thank to it we are able to give  upper bounds for  $r(I)$ and $c(I)$ in terms of generating degrees and the number of generators, see Theorem \ref{Lead}. These bounds are close to be sharp (see Proposition \ref{Red}), where we also show that in the worst case, the stability index $\gstab(I)$ of maximal generating degree (see (\ref{gstabD})  for the exact definition) of a monomial ideal $I$ should be at least a function of $d(I)$ with the number of variables in the exponent.

Our study of degree excess functions leads to some interesting consequences on the Castelnuovo-Mumford regularities $\reg(I^n)$ of $I^n$. It was independently proved in \cite{CHT} and \cite{Ko}, that $\reg(I^n) = p(I)n + e(I)$ for some $e(I)\geq 0$ and all $n\gg 0$. It is an open problem to  bound  $e(I)$ and the so-called index of stability $\regst(I)$ (see (\ref{regstD}) for the definition), even in the case of monomial ideals, see, e.g. \cite{Ber, C, EU}. We show that even in the case of monomial ideals, in the worst case, the invariants $e(I)$ and $\regst(I)$ should be at least exponential functions of the number of variables (Theorem \ref{Reg}). This could be a reason why the study of $e(I)$ and $\regst(I)$ is so hard.

The paper has two sections. In Section \ref{A}, we construct some simple degree excess functions. Then, using a quite simple observation to combine them (Lemma \ref{G5}), we are able to construct monomial ideals satisfying Theorem \ref{G6} and Theorem \ref{G7}.  In Section \ref{B}, we first use a geometric object, the Newton polyhedron, to determine $p(I)$ and to give bounds on  $r(I)$  and on $\epsilon(I;n)$ (Theorem \ref{Lead}). Then we construct monomial ideals to show that these bounds are close to be sharp. We conclude the paper with the proof of Theorem \ref{Reg}.

\section{Degree excess function} \label{A}

The symbols $\Zset,\ \Nset,\ \Qset$, $\Rset$ and $ \Rset_+$ denote the sets of integers, non-negative integers, rationals, real numbers and non-negative real numbers, respectively. Let $R := K[X_1,...,X_s]$ be a polynomial ring over a field $K$ and $ \mfr := (X_1,...,X_s)$. Consider the standard grading in $R$, that is $\deg(X_i) = 1$ for all $i=1,...,r.$  Given a  vector $\albf \in \Nset^r$ we write $\Xbf^\albf := X_1^{\alpha_1}\cdots X_r^{\alpha_r}$. Then $\deg(\Xbf^\albf) = |\albf| = \alpha_1+ \cdots + \alpha_r$. Let $I \subset R $ be a homogeneous non-zero ideal, which is minimally generated by homogeneous polynomials $\{f_1,...,f_t\}$.  The number 
$$d(I) = \max\{\deg(f_1),..., \deg(f_t)\} $$
is called the maximal generating degree of $I$. 
\vskip0.3cm

\noindent {\bf Kodiyalam's construction}: Let 
$$p(I) := \max\{ \deg(f_i)|\ f_i^n \not\in \mfr I^n \ \text{for all } n\geq 1,\ i\leq t\},$$
and 
$$J = (f_i |\  \deg(f_i) \leq p(I), \ \ i\leq t) = (I_{\leq p(I)}),$$
where $I_{\leq q}$  denotes the set of all homogeneous elements of $I$ of degrees at most $q$. Then $J$ is a reduction of $I$, i.e., $I^n= JI^{n-1}$ for $n\gg 0$.

The following result of Cutkosky-Herzog-Trung and Kodiyalam is a starting point of our study.

\begin{lem} \label{CHTK} {\rm  (\cite[Corollary 3 and Proposition 4]{Ko} and \cite[Theorem 3.1]{CHT}).} Let $I$ be a non-zero proper homogeneous ideal of $R$. Then
\begin{itemize}
\item[(i)] There is $q(I) \geq 0$ such that 
\begin{equation}\label{EKo}
p(I)n \leq d(I^n) \leq p(I) n+ q(I) \ \text{for all } n\geq 1.
\end{equation}
\item[(ii)] There is $q_0(I) \geq 0$ such that $d(I^n) = p(I) n+ q_0(I) $ for all $n\gg 0$.
\end{itemize} 
\end{lem}

Because of this lemma, we can make the following definition:

\begin{defn} We call $\epsilon (I; n) = d(I^n) - p(I)n \geq 0$ the {\it  degree excess function} of $I$. 
\end{defn}

If  $p(I) = d(I)$, then $\epsilon (I; n) =0$ for all $n$. However, in the general case, $\epsilon (I; n)$ could be quite complicate.

Recall that the reduction number of $I$ w.r.t. a reduction $J'$ is defined by
$$r_{J'}(I) = \min\{ n  \geq 0|\ I^{n+1} = J'I^{n} \}.$$
 In this paper we set
\begin{equation}\label{RedD}
r(I) := r_J(I),
\end{equation}
where $ J$ is the minimal reduction in the Kodiyalam's construction. 
Note that $I$ has many reductions. In the literature, $r(I)$ usually means the least number among the reduction numbers with respect to the so-called minimal reductions. This is not the case here, and $r(I)$ in the above definition  is in general much smaller than the usual one.

Below are some elementary properties of a degree excess function. 

\begin{lem} \label{G1} Let $J$ be as above and $r(I)$ the reduction number of $I$ w.r.t. $J$. Then
\begin{itemize}
\item[(i)] $\epsilon (I; n)$ is non-increasing for $n\geq r(I)$ and is a constant for $n\gg 0$.
\item[(ii)] $\epsilon (I; n)\leq \min\{n, r(I)\} (d(I) - p(I))$ for all $n\geq 1$.
\item[(iii)] For all $n\leq r(I)$,  $\epsilon (I; n) \geq n$.
\item[(iv)] For all $m\leq n$, $n\epsilon (I; m) \geq m\epsilon (I; n)$.
\end{itemize} 
\end{lem}
 
 \begin{proof} The fact that $\epsilon (I; n)$ is  a constant for $n\gg 0$ is a consequence of Lemma \ref{CHTK}(ii). The rest of (i) and (ii) follows from the equality $I^{n} = J^{n-r(I)}I^{r(I)}$ for all $n\geq r(I) + 1$ and $d(J) = p(I)$.
 
 (iii) Let $f\in I^n \setminus J I^{n-1}$ be a minimal generator of $I^n$. One can write 
 $f = \sum_k f_{k1}\cdots f_{k n} g_k,$ where $f_{ki} \in \{f_1,...,f_t\}$ for all $k,i$, $g_k$ is a homogeneous polynomial and $\deg(f_{k1} ) \geq \cdots \geq \deg(f_{kn} )$. Since $f\not\in J I^{n-1}$ there is $k$ such that $\deg(f_{k1} ), ...,  \deg(f_{kn} ) > p(I)$. Then 
 $$d(I^n) \geq \deg(f) \geq \deg(f_{k1} ) + \cdots +  \deg(f_{kn} ) \geq n(p(I) + 1),$$
 which implies $\epsilon (I; n) \geq n$.
 
(iv) Let $f\in I^n$ be a minimal generator such that $\deg(f) = d(I^n)$. As above we write
 $f = \sum_k f_{k1}\cdots f_{k n} g_k,$ where $f_{ki} \in \{f_1,...,f_t\}$ for all $k,i$, $g_k$ is a homogeneous polynomial and $\deg(f_{k1} ) \geq \cdots \geq \deg(f_{kn} )$.  Since $f$ is a minimal generator, there is $k$ such that $g_k=1$ and for all $m<n$, $f_{k1}\cdots f_{k m}$ is a minimal generator of $I^m$. Hence
 $$\begin{array}{ll}
 n[ p(I)m + \epsilon (I; m)] & = n d(I^m) \geq n[ \deg(f_{k1} ) + \cdots + \deg(f_{km} ) ] \\ & \geq m[  \deg(f_{k1} ) + \cdots + \deg(f_{kn} )] = m \deg(f) \\ & = m d(I^n)  = m[p(I)n + \epsilon (I; n) ],
\end{array}  $$
 which yields $n\epsilon (I; m)\geq m\epsilon (I; n)$.
\end{proof}

In order to understand the degree excess functions, the following questions/ problems are quite natural to ask for.
\vskip0.3cm

{\bf Problems}. 1. Characterize degree excess functions. In particular, how many local maxima can the function $\epsilon (I; n)$ have? 

2. When does $\epsilon (I; n)$ become a constant, or equivalently give a bound on the index of stability of the maximal generating degree defined by
\begin{equation}\label{gstabD}
\gstab (I) = \min\{n_0|\ \epsilon (I; n)= \epsilon (I; n_0) \ \ \text{for all}\ n\geq n_0\}.
\end{equation}

3. Give a bound on $r(I)$.
\vskip0.5cm

In the rest of paper we  give some partial solutions  to these problems.  Note that Vasconcelos already showed that for any homogeneous ideal of $R$ there is a bound on any reduction number, but this bound  is a  non-elementary function with four levels of exponentiation (see \cite[Proposition 2.3]{V}). We believe that for the reduction number $r(I)$ in (\ref{RedD}) the bound should be much smaller. However, even in the case of monomial ideals, the bound in the worst case should be at least an exponential function, see Proposition \ref{Red}.

In the sequel we mainly consider monomial ideals. Recall that  a monomial ideal of $R$ is an ideal generated by monomials. If $I$ is a monomial ideal, $I$ has the unique minimal generating system consisting of monomials which is denoted by $G(I)$. An element of $G(I)$ is called a {\it (monomial) minimal generator} of $I$.  

\vskip0.3cm
\noindent {\bf Notation.} We often consider the following special monomials in two variables:
$$m_i  = X_1^{p-i}X_2^i, \ \ \text{where}\ p\geq 4,\  0\leq i\leq p.$$
Of course, these monomials depend on the exponent $p$. However, this number is clearly given in the context, so that we omit it for the sake of simplicity.

\begin{rem} \label{EG2} Assume that a monomial ideal $\afr$ is a reduction of $I$ and $f\in R$ is a monomial such that $I=(\afr ,f)$. Then
$$r_{\afr}(I) = \min\{ t\geq 1|\ f^t \in \afr^t \} - 1.$$
\end{rem}

The following result shows that the bounds in Lemma \ref{G1} are sharp.

\begin{lem} \label{G2} Given $p\geq 4$ and $2 \leq a \leq p-2$ such that $p-1$ is not divisible by $a$. Let 
 $$J_{p,a} : = (m_0, m_p , m_{p-1} , X_1^{p-1} X_2^a) \subset K[X_1,X_2].$$
 Then   $ r(J_{p,a}) = r:= \lfloor \frac{p-1}{a}\rfloor$ and
 $$ \epsilon(J_{p,a}; n) = \begin{cases} 
  n(a-1) &  \ \text{if}\  1 \leq n  \leq r-1,\\
  r(a-1) & \ \text{if}\ n \geq r.
\end{cases}$$
\end{lem}

\begin{proof} 
(i) For short, we set $m= X_1^{p-1} X_2^a$,  $I= J_{p,a}$ and $J= (m_0,m_p,m_{p-1})$.  Then $J$ is the reduction of $I$ in the Kodiyalam's construction, and $p(I) = p$.

Note that $r+1 >  p/a$, whence $(r+1) a >  p$, or equivalently $(r+1)a - p >  0$. Since $a\geq 2$, $ p\ge r+1$. Hence 
$$\begin{array}{ll} 
m^{r+1}&  = X_1^{(r+1)(p-1)}  X_2^{(r+1)a} = X_1^{rp} X_2^p  X_1^{p-(r+1)} X_2^{(r+1)a-p}  \\
& = m_0 m_p^r  X_1^{p-(r+1)} X_2^{(r+1)a-p}  \in J^{r+1} .
\end{array}$$
 By Remark \ref{EG2}, $r(I) \leq  r$, and 
 \begin{equation}\label{EG20}
 d(I^n) \leq d(J^{n-r}) + d(I^r) \leq np + r(a-1) \ \text{for all}\  n\geq r. 
\end{equation}
 
(ii)  Further, assume that $m^t \in  J^t$ for some  $0< t\leq r$.  Then $m^t = m' m_0^{\alpha_0}m_p^{\alpha_p} m_{p-1}^{\alpha_{p-1}}$ for a monomial $m'$ and some non-negative integers $\alpha_i$ such that $\alpha_0 + \alpha_p + \alpha_{p-1} = t$. Looking at the exponents of $X_1$ on both sides, we get  $(p-1)t \geq p\alpha_0$, whence $\alpha_0 < t$. Hence $\alpha_p+\alpha_{p-1} \geq 1$. Looking at the exponents of $X_2$, we now obtain 
\begin{equation}\label{EG2b} at \geq p\alpha_p + (p-1)\alpha_{p-1} \geq (p-1)(\alpha_p+\alpha_{p-1})\geq p-1.
\end{equation}
Since $p-1$ is not divisible by $a$, $r < (p-1)/a$, whence $ar < p-1$. Hence $p-1> ar \geq at \geq p-1$, a contradiction. Thus,  
 \begin{equation}\label{EG21}
 m^t \not\in  J^{t}\  \text{for all} \ 0<t\leq r.
\end{equation}
 By Remark \ref{EG2} this implies $r(I) \geq  r $, which then yields $r(I) = r$.  

(iii)  Now let $t\leq r$ and $n\geq t$.  Assume that $m^t m_0^{n-t} \in \mfr I^n$. 
 Then one can write
 $$m^t m_0^{n-t} = m' m_0^{\alpha_0}m_p^{\alpha_p}m_{p-1}^{\alpha_{p-1}}m^\alpha,$$ 
 where $m'\neq 1$ is a monomial and $\alpha+ \alpha_0 + \alpha_p + \alpha_{p-1} = n$. Comparing the total degrees in both sides we must have $\alpha < t$. Therefore
  $$m^u m_0^{n-t} = m' m_0^{\alpha_0}m_p^{\alpha_p}m_{p-1}^{\alpha_{p-1}},$$ 
  where $1\leq u := t - \alpha \leq r$. By virtue of (\ref{EG21}), $\alpha_0 < n-t$. Set $v= n-t -\alpha_0>0$. Then
  $$m^u m_0^v = m' m_p^{\alpha_p}m_{p-1}^{\alpha_{p-1}},$$ 
  where $\alpha_p + \alpha_{p-1} = u+v$. Comparing the exponents of $X_2$ in both sides, we get 
  $au \geq \alpha_p p + \alpha_{p-1} (p-1)\geq (p-1)(\alpha_p + \alpha_{p-1})$. Since $au\leq ar < p-1$, we must have $\alpha_p = \alpha_{p-1} = 0$. Then $u+v = 0$, a contradiction. 
  
  Summing up, we have shown that  $m^t m_0^{n-t} $ is a minimal generator of $I^n $. Therefore $d(I^n) \geq t(p+a-1) + (n-t)p = np + t(a-1)$. Since $d(I^t) \leq t(p+a-1)$ for all $t$, together with (\ref{EG20}), this implies $d(I^t) = tp + t(a-1)$ for all $t\leq r$ and  $d(I^n)  = np + r(a-1)$ for all $n\geq r$, or equivalently $\epsilon(I; n) = n(a-1)$ for $n\leq r$ and $\epsilon(I;n) = r(a-1)$ for $n\geq r$.
\end{proof}

The following result is a folklore. We give a proof for the sake of completeness.

\begin{lem} \label{G3} For  given $p\geq 4$ and $2\leq i \leq \lfloor p/2\rfloor$, let
$$\begin{array}{ll}
I_p &= (m_0,m_p,m_1,m_{p-1}), \ \text{and} \\
I_{p,i,0} &= (m_0, m_p,m_1, m_{p-1}, m_i),
\end{array}$$
be two ideals of $K[X_1,X_2]$. Then
\begin{itemize}
\item[(i)] For all $n\geq p-2$,
$$(I_{p,i,0})^n = (I_p)^n = (\Xbf^{\albf}| \ |\albf| = pn).$$
\item[(ii)] For all $u\leq i-2$ and $v\leq p-i-2$ we have  $m_im_0^um_p^v \not\in (I_p)^{1+u+v}$. 
\item[(iii)] $m_i^t \not\in I_p^t$ for all $1\leq t < \frac{p-1}{i}$.
\end{itemize} 
\end{lem}

\begin{proof} (i) Clearly
$$(I_p)^n \subseteq (I_{p,i,0})^n \subseteq (\Xbf^{\albf}| \ |\albf| = pn).$$
We first prove the reverse inclusions in the case $n= p-2$. Let $\alpha_1, \alpha_2 \in \Nset$ such that $\alpha_1 + \alpha_2 = p(p-2)$. Let $\alpha_1 = ap+j$, where $0\leq j < p$ and $a\leq p-2$. If $j\leq p-a -2$, then $p-2-a-j \geq 0$, and
$$\begin{array}{ll}
X_1^{\alpha_1}X_2^{\alpha_2} &= (X_1X_2^{p-1})^j(X_1^p)^a (X_2^p)^{p-2-a-j} \\
&= m_{p-1}^jm_0^a m_p^{p-2-a-j} \in (I_p)^{p-2}.
\end{array}$$
Assume that $j \geq p-a-1$. Then $a\leq p-3$,  $a+j+1-p \geq 0$ and
$$X_1^{\alpha_1}X_2^{\alpha_2} = m_1^{p-j} m_0^{a+j+1-p} m_p ^{p-3-a}  \in (I_p)^{p-2}.$$
In both cases $X_1^{\alpha_1}X_2^{\alpha_2}  \in (I_p)^{p-2}$, which yields $(\Xbf^{\albf}| \ |\albf| = p (p-2)) \subseteq (I_p)^{p-2}$, whence
$$(I_p)^{p-2} = (I_{p,i,0})^{p-2} =  (\Xbf^{\albf}| \ |\albf| = p (p-2) ).$$
For $n> p-2$, by induction we have 
$$(\Xbf^{\albf}| \ |\albf| = pn) = (X_1^p, X_2^p)(\Xbf^{\albf}| \ |\albf| = p(n-1)) = (m_0,m_p) (I_p)^{n-1} \subseteq (I_p)^n.$$
This implies (i).
\vskip0.3cm

(ii) Let  $u\leq i-2$ and $v\leq p-i-2$. Assume that
\begin{equation}\label{EG30}
m_i m_0^um_p^v = m_0^{\alpha_0}m_p^{\alpha_p}m_1^{\alpha_1}m_{p-1}^{\alpha_{p-1}} ,\ \alpha_0 + \alpha_1 + \alpha_p + \alpha_{p-1} = 1+u+v.
\end{equation} 
Note that $\alpha_k \leq p-3$ for all $k= 0,1, p, p-1$. Comparing the exponents of $X_1$ and $X_2$ in both sides of the above equality, we get
\begin{eqnarray}
\label{EG31} p-i+pu &= & \alpha_0 p  + \alpha_1 (p-1) + \alpha_{p-1} = (\alpha_0 + \alpha_1)p + \alpha_{p-1} - \alpha_1,\\
\label{EG32} i +pv &=& \alpha_1 + \alpha_{p-1}(p-1) + \alpha_pp = (\alpha_p + \alpha_{p-1})p + \alpha_1 - \alpha_{p-1}.
\end{eqnarray}

{\it Case 1}. $\alpha_1 \geq \alpha_{p-1}$. Since $\alpha_1<p$, from (\ref{EG32}) it follows that $\alpha_1 - \alpha_{p-1} = i $. In particular $\alpha_1 \geq i$. Then (\ref{EG31})  gives $i -1 \geq u+1 = \alpha_0 + \alpha_1 \geq i$, a contradiction.

{\it Case 2}. $\alpha_1 < \alpha_{p-1}$. Then, by (\ref{EG31}) $\alpha_{p-1} - \alpha_1 = p-i$, whence $\alpha_{p-1} \geq p-i $, and by (\ref{EG32}), $p-i-2 \geq v = \alpha_p + \alpha_{p-1} - 1 \geq p-i -1$, a contradiction. 

Summing up, (\ref{EG30}) cannot happen.

\vskip0.3cm

(iii) The proof of this is similar to the proof of  Lemma \ref{G2}. 
\end{proof}

By Lemma \ref{G1}(i), $\epsilon(I;n)$ is non-increasing for $n\geq r(I)$. One may expect that after $r(I)$, the function $\epsilon(I;n)$ becomes stable (i.e., equals a constant) very soon. However, using the above result we can now construct ideals $I$ with small reduction number $r(I)$ and big index of stability $\gstab(I)$. It shows that, in general,  the reduction number $r(I)$ and the number of variables are not enough to bound $\gstab(I)$. A bound on $\gstab(I)$ must depend on the generating degrees of $I$. The fact that $r(I_{p,2,a}) = \lceil \frac{p-1}{2} \rceil - 1$ was pointed out by the referee.

\begin{lem} \label{G4}  Given $p\geq 4$, $2\leq i \leq \lfloor p/2\rfloor$ and $a\geq 1$. Let 
$$I_{p,i ,a} = (m_0,m_p,m_1,m_{p-1}, m_i X_3^a) \subset K[X_1,X_2,X_3].$$
Then $I_p = (m_0,m_p,m_1,m_{p-1})$ is a reduction of $I_{p,i,a}$, and
\begin{itemize}
\item[(i)] $r(I_{p,i,a}) \geq \lceil \frac{p-1}{i} \rceil - 1$. The equality holds if $i=2$.
\item[(ii)] $\epsilon(I_{p,i,a}; n) =  an$ for $1\leq n \leq  \lceil \frac{p-1}{i} \rceil - 1$ and $\epsilon(I_{p,i,a}; n) = 0$  for $n\geq p-2$.
\item[(iii)] (The case $i= \lfloor p/2\rfloor$): $r(I_{p,\lfloor p/2\rfloor, a}) = 1$, and
$$\epsilon(I_{p,\lfloor p/2\rfloor, a}; n) = \begin{cases}  a & \ \text{if}\ 1\leq n\leq p-3,\\
0  & \ \text{if}\ n\geq p-2.
\end{cases}$$
\end{itemize} 
\end{lem}

\begin{proof} The fact that $I_p = (m_0,m_p,m_1,m_{p-1})$ is a reduction of $I_{p,i,a}$ in the Kodiyalam's construction is clear. Hence $p(I_{p,i,a}) = p$.

(i) Assume that $(m_iX_3^a)^t \in (I_p)^t$ for some $t\geq 1$. Since no generator of $I_p$ contains $X_3$, $m_i^t \in (I_p)^t$. By Lemma \ref{G3}(iii), $t \geq \frac{p-1}{i}$, whence $t\geq \lceil \frac{p-1}{i} \rceil$. Since $ I_{p,i,a} = (I_p, m_iX_3^a)$, by Remark \ref{EG2},  $r( I_{p,i,a}) \geq \lceil \frac{p-1}{i} \rceil - 1$.

Let $i=2$. If $p= 2q$, then $m_2^q = X_1^{q(2q-2)}X_2^{2q} = m_0^{q-1}m_p$. If $p= 2q+1$, then $m_2^q = X_1^{q(2q-1)}X_2^{2q} = m_0^{q-1}m_{p-1}$. In both cases $m_2^q \in I_p^q$. By Remark \ref{EG2}, $r(I_{p,2,a}) \leq q = \lceil \frac{p-1}{2} \rceil - 1$, whence $r(I_{p,2,a}) = \lceil \frac{p-1}{2} \rceil - 1$.
\vskip0.3cm

(ii)  Assume that $n\geq p-2$. By Lemma \ref{G3}(i)  we have
$$  (I_p)^n \subseteq (I_{p,i ,a})^n \subseteq (I_{p,i,0})^n =   (I_p)^n .$$
Here we consider the ideals $I_p$ and $I_{p,i,0}$ defined in Lemma \ref{G3} as  ideals in $K[X_1,X_2,X_3]$. Hence $(I_{p,i,a})^n  =   (I_p)^n$, which yields $d((I_{p,i,a})^n)  =   d((I_p)^n) = pn$, or equivalently $\epsilon (I_{p,i,a}; n) =0$ for all $n\geq p-2$.

Further, let $t\leq \lceil \frac{p-1}{i} \rceil - 1$. Assume that $(m_iX_3^a)^t \in \mfr (I_{p,i,a})^t$. Then
$$(m_iX_3^a)^t = m' m_0^{\alpha_0} m_p^{\alpha_p} m_1^{\alpha_1}m_{p-1}^{\alpha_{p-1}}(m_iX_3^a)^\alpha,$$
for some $\alpha, \alpha_i\in \Nset$ such that $\alpha_0+ \alpha_p+ \alpha_1+ \alpha_{p-1} + \alpha = t$ and $m'\neq 1$ is a monomial. Comparing the total degrees  in both sides we must have $\alpha< t$. But then $m_i^{t-\alpha} \in (I_p)^{t- \alpha}$, which is impossible by Lemma \ref{G3}(iii). Hence  $(m_iX_3^a)^t$ is a minimal generator of $(I_{p,i,a})^t$, which yields $\epsilon(I_{p,i,a}; t) \geq at$, whence $\epsilon(I_{p,i,a}; t) = at$.

(iii) Let $p = 2q + \varepsilon$, $\varepsilon= 0,1$.  This means $q = \lfloor p/2\rfloor$. Since 
$$(m_q X_3^a)^2 = X_1^{2(q+\varepsilon)} X_2^{2q} X_3^{2a} =
 \begin{cases} m_0m_p X_3^{2a} &  \ \text{if}\ \varepsilon = 0,\\
 m_0m_{p-1} X_3^{2a} &  \ \text{if}\ \varepsilon = 1,
\end{cases}$$
$(m_qX_3^a)^2 \in (I_p)^2$. Hence $(I_{p,q,a})^n = (I_p)^{n-1}I_{p,q,a}$ for all $n\geq 2$. This means $I_p$ is a reduction of $I_{p,q,a}$ and $r(I_{p,q,a}) = 1$. This also implies 
\begin{equation}\label{EG40}
d((I_{p,q,a})^n ) \leq p(n-1) + p+a = pn +a.
\end{equation} 

Let  $u \leq q-2$ and $v \leq p-q-2$. Assume that $(m_q X_3^a)m_0^um_p^v \in \mfr (I_{p,q,a})^{1+u+v}$. Then
\begin{equation}\label{EG41}
(m_q X_3^a)m_0^um_p^v  = m' m_0^{\alpha_0} m_p^{\alpha_p} m_1^{\alpha_1}m_{p-1}^{\alpha_{p-1}}(m_qX_3^a)^{\alpha_q},
\end{equation}
where $1\neq m'$ is a monomial of $K[X_1,X_2,X_3]$, $\alpha_i\in \Nset$ and $\alpha_0 + \alpha_p + \alpha_1 + \alpha_{p-1} + \alpha_q = 1 + u+ v$.  
A comparison of the  total  degrees  in two sides gives $\alpha_q =0$, and (\ref{EG41}) becomes
$$m_q m_0^um_p^v  = m" m_0^{\alpha_0} m_p^{\alpha_p} m_1^{\alpha_1}m_{p-1}^{\alpha_{p-1}} \in (I_p)^{1+u+v},$$
which is impossible by virtue of Lemma \ref{G3}(ii). This says that $(m_qX_3^a)m_0^um_p^v$ is a minimal generator of $(I_{p,q,a})^{1+u+v}$, whence $d((I_{p,q,a})^{n}) \geq  pn +a$ for all $n\leq 1 + q-2 + p-q -2  = p-3$. Together with (\ref{EG40}), we get $d((I_{p,qa})^{n}) =  pn +a$, and so $\epsilon(I_{p,qa}; n) = a$ for all $n\leq p-3$.
\end{proof} 
 
Using the following simple observation and the ideals in Lemmas \ref{G2} and \ref{G4}, one can construct ideals with more complicate degree excess functions. The following proof is suggested by the referee, which is much simpler than the original one.  

\begin{lem} \label{G5}  Let $I \subset K[\Xbf]$ and $J \subset K[\Ybf]$ be non-zero proper homogeneous ideals, where $\Xbf$ and $ \Ybf$ are disjoint sets of variables. Assume that $\{f_1,...,f_t\}\subset K[\Xbf]$ and $\{ g_1,...,g_u \} \subset K[\Ybf]$ are homogeneous polynomials which form minimal bases of $I$ and $J$, respectively. Then the set $\{f_ig_j| \ i=1,...,t; j= 1,...,u \}$ forms a minimal basis of the product $(IJ) : = IJK[\Xbf, \Ybf]$. It implies that the  degree excess functions  has the following property
$$\epsilon(IJ; n) = \epsilon(I;n) + \epsilon(J;n).$$
\end{lem}

\begin{proof} Applying the first claim to $I^n$ and $J^n$ for all $n$ we immediately get $d((IJ)^n) = d(I^n) + d(J^n)$, which then implies $\epsilon(IJ; n) = \epsilon(I;n) + \epsilon(J;n).$

In order to prove the first claim, denote $\mfr = (\Xbf)$, $\nfr = (\Ybf)$ and $T= K[\Xbf, \Ybf]$. Note that
$$ \frac{I}{\mfr I} \otimes_K \frac{J}{\nfr J}  \cong \frac{I\otimes_K J}{I \otimes_K (\nfr J) + (\mfr I) \otimes_K J}
\cong \frac{IJ}{I\nfr J + \mfr IJ} = \frac{IJ}{(\mfr T + \nfr T) IJ}.$$
Therefore $\mu(IJ) = \mu(I) \mu(J)$, where $\mu (.)$ denotes the minimal number of generators. Since the set $\{f_ig_j| \ i=1,...,t; j= 1,...,u \}$ generates $IJ$, it must be a minimal basis of $IJ$.
\end{proof}

The following two results say that a degree excess function could be quite complicate.

\begin{thm} \label{G6}  Let $f:\ \Nset \rightarrow \Nset$ be any non-increasing function. Then there is a monomial ideal $I$ such that $\epsilon(I;n) = f(n)$ for all $n\geq 1$.
\end{thm}

\begin{proof} Assume that
$$f(n) = \begin{cases} d_1  &  \ \text{if} \ 1 = n_1 \leq n < n_2, \\
  d_2  &  \ \text{if} \ n_2 \leq n < n_3, \\
  & ... \\
  d_{k-1}  &  \ \text{if} \ n_{k-1} \leq n < n_k, \\
  d_k  &  \ \text{if} \  n\geq n_k, 
\end{cases}$$
where $d_1 > d_2 > \cdots > d_k$ and $1= n_1< n_2 < \cdots < n_k$. If $k=1$ we may assume that $d_1>0$. Then the ideal $I = J_{d_1+3, d_1+1}$ defined in Lemma \ref{G2} has $\epsilon(I,n) = d_1$ for all $n\geq r(I) = \lfloor (d_1+2)/ (d_1+1)\rfloor = 1 = n_1$. 

Let $k\geq 2$. Set $p_i = n_{i+1} + 2$ and $a_i = d_i - d_{i+1}  \geq 1$, where $i= 1,..., k-1$. Consider the ideals $I_{p_i, \lfloor p_i/2\rfloor, a_i}$ of the ring $K[\Xbf_{(i)}] := K[X_{i1},X_{i2},X_{i3}] $ defined in Lemma \ref{G4}, where $\Xbf_{(1)}, ..., \Xbf_{(k-1)}$ are disjoint sets of variables. Let
$$I = \prod_{i=1}^{k-1} I_{p_i, \lfloor p_i/2\rfloor , a_i} \subset K[\Xbf_{(1)}, ..., \Xbf_{(k-1)}].$$
Assume that $n_j\leq n< n_{j+1}$ for some $j=1,...,k$, where $n_{k+1} := \infty$. If $i\leq j-1$, then $p_i - 2= n_{i+1} \leq n_j \leq n$. By Lemma \ref{G4}, $\epsilon (I_{p_i, \lfloor p_i/2\rfloor , a_i}; n) = 0$. If $i\geq j$, then $p_i - 2= n_{i+1} \geq n_{j+1} >  n$. By Lemma \ref{G4}, $\epsilon (I_{p_i, \lfloor p_i/2\rfloor , a_i}; n) = a_i $. Therefore, by Lemma \ref{G5},
$$\begin{array}{ll}
\epsilon (I;n) & = \sum_{i=1}^{k-1} \epsilon(I_{p_i,\lfloor p_i/2\rfloor , a_i}; n) = a_j + a_{j+1} + \cdots + a_{k-1} \\
&= (d_j-d_{j+1}) + (d_{j+1} - d_{j+2}) + \cdots + (d_{k-1} - d_k) = d_j - d_k.
\end{array}$$
Thus,  $\epsilon(I;n) = f(n) - d_k$ for all $n\geq 1$. If $d_k = 0$ then we are done. If $d_k>0$, we set $I' = IJ \subset K[\Xbf_{(1)}, ..., \Xbf_{(k)}]$, where $J = J_{d_k+3, d_k+1}$ is the ideal of $K[\Xbf_{(k)} := K[X_{k1}, X_{k2}]$ defined in Lemma \ref{G2}. By this lemma,  $r(J) = 1$ and $\epsilon (J;n) = d_k$ for all $n\geq 1$. Using again Lemma \ref{G5}, we then get
$$\epsilon (I'; n) = \epsilon(I;n) + \epsilon(J;n) = f(n) - d_k + d_k = f(n),$$
for all $n\geq 1$.
\end{proof}

\noindent {\it Remark}. If $r(I) =0$, then $\epsilon(I;n) = 0$ for all $n\geq 1$. If $r(I) = 1$, by Lemma \ref{G1}, $\epsilon(I; n) $ is non-increasing for all $n$. Both ideals $I$ and $I'$ in the proof of the above theorem have reduction number one. So Theorem \ref{G6} implies that a numerical function $f: \Nset \rightarrow \Nset$ is a degree excess functions of a homogeneous ideal $I$ with $r(I) = 1$ if and only if it is non-increasing.
\vskip0.5cm

In Lemma \ref{G2} we already construct ideals with a non-decreasing degree excess function. Combining this lemma with Lemma \ref{G5}, one can construct ideals with more complicate non-decreasing degree excess functions. However, unlike the case of non-increasing functions, due to the restriction given in Lemma \ref{G1}(iv), not any convergent non-decreasing function could be realized as a degree excess function (of a homogeneous ideal).  For an example, if $f(n) = 1$ and $f(n+1) > 1$ for some $n\geq 2$, then $f(n)$ cannot be a degree excess function. However, the following result says that the degree excess function could get  local maxima at any given collection of positions. Here, a numerical function $f:\ \Nset \rightarrow \Zset$ is said to have a local maxima at $n\geq 2$ if $f(n) > f(n-1)$ and $f(n) > f(n+1)$. It is of interest to mention that, inspired of  the  comparison of numbers of generators of $I$ and $I^2$ in \cite{HSZ}, it was shown in \cite{AHZ} that given a number $q$ there is a monomial ideal $I$ with the property that the function which describes the number of generators of $I^n$ has at least $q$ local maxima.

\begin{thm} \label{G7} Given any sequence of positive numbers $2 \leq n_1 < n_2 < \cdots < n_k$ ($k\geq 1$) such that $n_{i+1}- n_i \geq 2$. Then there is a monomial ideal such that its degree excess function gets  local maxima exactly at points of the set $\{n_1,...,n_k\}$.
\end{thm}

\begin{proof} Let $J= J_{2(n_k+2), 2}$ be the ideal of $K[Y_1,Y_2]$ defined in Lemma \ref{G2}. Then $r(J) = \lfloor (2n_k+3)/ 2\rfloor = n_k +1$, and
$$\epsilon (J; n) = \begin{cases}  n & \ \text{if}\ n\leq  n_k, \\
  n_k + 1& \ \text{if}\ n\geq n_k +1.
\end{cases}$$
Consider the following function
$$f(n) = 2(k+1-i) \ \text{if} \ n_{i-1}< n\leq n_i, \ (i=1,...,k+1),$$
where we set $n_0 = 0$ and $n_{k+1} := \infty$. By Theorem \ref{G6}, there is a monomial ideal $I \subset K[X_1,...,X_s]$ such that $\epsilon(I;n) = f(n)$ for all $n\geq 1$. Consider the product $IJ \subset K[X_1,...,X_s,Y_1,Y_2]$. By Lemma \ref{G5}, we have
$$\begin{array}{ll}
\epsilon (IJ; n_i) &= f(n_i) + \epsilon(J,n_i) = 2(k+1-i) + n_i = 2(k-i) + n_i + 2,\\
\epsilon (IJ; n_i-1) &= f(n_i-1) + \epsilon(J,n_i-1) = 2(k+1-i) + n_i - 1= 2(k-i) + n_i + 1,\\
\epsilon (IJ; n_i+1) &= f(n_i+1) + \epsilon(J,n_i+1) = 2(k+1-i-1) + n_i + 1= 2(k-i) + n_i + 1,
\end{array}$$
for any $i=1,...,k$. Thus, $n_i$ is a local maximum point of $\epsilon(IJ; n)$. Since $f(n)$ is a constant and $\epsilon(J; n)$ is  non-decreasing in any interval $n_{i-1}< n< n_i $, where $i=1,...,k+1$, there is no other local maximum points of $\epsilon(IJ; n)$.
\end{proof}

\noindent {\it Remark}. In the proof of Theorem \ref{G7}, if instead of $J_{2(n_k+2),2}$ we take $J = I_{2n_k+4, 2, 1}$ defined in Lemma \ref{G4}, then by Lemma \ref{G4} and Lemma \ref{G1}(i), we get 
$$\epsilon (J; n) = \begin{cases}  n & \ \text{if}\ n\leq  n_k +1, \\
  & \ \ \text{non-increasing for}\ n_k+1 \leq  n \leq 2n_k + 1,\\
 0  &  \   \text{if}\ n\geq 2n_k+2.
\end{cases}$$
With the same choice of $I$, then $\epsilon (IJ; n)$ gets local maxima at points $\{n_1,...,n_k\}$. 

This  degree excess function  has a given number of local maxima and drops to $0$ when $n\gg 0$, while the one  in the proof of Theorem \ref{G7} has  a big limit value.

\section{Maximal generating degrees of powers of a monomial ideal} \label{B}

In this section we consider powers of monomial ideals. Using a  geometric interpretation, we give an explicit way to compute the coefficient $p(I)$ in the Kodiyalam's construction. Moreover we can give a bound on the reduction number $r(I)$. 

For a subset $A \subseteq R$, the exponent set of $A$ is $E(A) := \{\albf \mid \Xbf^{\albf}\in A\} \subseteq  \Nset^s$. Recall that the Newton polyhedron of a monomial ideal $I$ is $NP(I) := \conv (E(I))$, the convex hull of the exponent set of $I$ in
the space $\Rset^s$. If $G(I)= \{\Xbf^{\albf_1}, ..., \Xbf^{\albf_t}\}$ is the set of minimal generating monomials of $I$, then $E(G(I))= \{\albf_1,...,\albf_t\}$ and
$$NP(I) := \conv(\albf_1,...,\albf_t)+ \Rset_+^s.$$
A useful property of $NP(I)$ is that it is a pointed polyhedron, that is $NP(I)$ contains no non-trivial linear subspace of $\Rset^s$. Hence each minimal face of $NP(I)$ consists of just one point  which is called a vertex of $NP(I)$. The set of vertices of $NP(I)$, denoted by $V(I)$, is uniquely defined by $NP(I)$ and is a subset of $E(G(I))$. The decomposition theorem of polyhedra says that
\begin{equation}\label{IC4}
NP(I) = \conv (V(I)) + \Rset_+^s,
\end{equation}
(see, e.g., (29) on p. 106 in \cite{Sch}). Algebraically, $V(I)$ is the minimal subset of $E(G(I))$ satisfying the equation (\ref{IC4}). In the sequel, the unit vectors of $\Rset^s$ are denoted by $\ebf_1,...,\ebf_s$. 

\begin{exm} \label{Ex1} Let $I = (X_1^2, X_2^3, X_1X_2^2X_3)$. Since
$$(1,2,3) = \frac{1}{2}(2,0,0) + \frac{1}{2}(0,3,0) + \frac{1}{2}\ebf_2 + \ebf_3,$$
$V(I) = \{(2,0,0), (0,3,0)\}$, while $E(G(I) ) = \{ (2,0,0), (0,3,0), (1,2,1)\}$. The  presentation of $(1,2,3)$ in $\conv (V(I)) + \Rset_+^s$ is not unique. For an example, here is another presentation:
$$(1,2,3) = \frac{1}{3}(2,0,0) + \frac{2}{3}(0,3,0) + \frac{1}{3}\ebf_1 + \ebf_3.$$
\end{exm}

The set $V(I)$ was used in \cite{H3} to study the maximal generating degrees of the integral closure $\Inba$ of $I^n$. We set
$$\delta(I) = \max\{ |\vbf| |\ \vbf \in V(I) \} .$$
In \cite{H3} it is shown that $d(\Inba) = \delta(I) n + c$ for all $n\gg 0$, where $c\in \Nset$ is a  constant. We now show that like the case of integral closures, this number $\delta(I)$ is exactly $p(I)$.
 
  Since  $NP(I^n) = nNP(I)$, the set of vertices of $NP(I^n)$ is the set of  $n$-multiples of vertices of $V(I)$, i.e. $V(I^n) = nV(I) $. On the other hand, by the definition,   $ V(I^n) \subseteq E(G(I^n))$. Hence
  
\begin{lem} \label{MG1} For all $n>0$, $d(I^n) \geq \delta (I) n$.
\end{lem}

We need the following technical result from Linear Programming.

\begin{lem} \label{LP} Let $U$ be a finite subset of non-negative integral vectors in $\Nset^s$ of lengths at most $\delta$ and $\vbf \in \Nset^s$.  Assume that $\vbf \in \conv (U) + \Rset^s_+$. Then one can find a positive integer  $N\leq s^{s/2}\delta^s$, non-negative integers  $\alpha_1,...,\alpha_m, \beta_1,...,\beta_s$  and vectors $\ubf_1,...,\ubf_m\in U$ such that $\sum \alpha_i = N$ and 
$$N \vbf = \sum_{i=1}^m \alpha_i \ubf_i +  \sum_{j=1}^s \beta_j \ebf_j.$$
\end{lem}

\begin{proof} By Carath\'eodory's Theorem, one can find $m\leq s+1$ points $\ubf_1,...,\ubf_m\in U$ and non-negative numbers  $\alpha'_1,...,\alpha'_m, \beta'_1,...,\beta'_s$  such that $\sum \alpha'_i = 1$,  and 
\begin{equation}\label{ELP0}
 \vbf = \sum_{i=1}^m \alpha'_i \ubf_i +  \sum_{j=1}^s \beta'_j \ebf_j.
\end{equation}
Consider the following system of linear constrains
\begin{equation} \label{ELP1}
 \begin{cases}
(u_{11}- u_{m1})x_1 + \cdots + (u_{(m-1)1} - u_{m1})x_{m-1} + y_1 = v_1,\\
\cdots  \\
(u_{1s}- u_{ms})x_1 + \cdots + (u_{(m-1)s} - u_{ms})x_{m-1} + y_s = v_s,\\
 x_1 + \cdots + x_{m-1} \leq  1,\\
x_1,...,x_{m-1}, y_1,...,y_s \geq  0.
\end{cases}
\end{equation}
Then $\alpha'_1,...,\alpha'_{m-1}, \beta'_1,...,\beta'_s$ is a feasible solution of the system (\ref{ELP1}). Note that $\alpha'_1,...,\alpha'_{m-1}$ is a feasible solution of the following system of linear constrains of $m-1$ variables:
\begin{equation} \label{ELP2}
 \begin{cases}
(u_{11}- u_{m1})x_1 + \cdots + (u_{(m-1)1} - u_{m1})x_{m-1}  \leq v_1,\\
\cdots  \\
(u_{1s}- u_{ms})x_1 + \cdots + (u_{(m-1)s} - u_{ms})x_{m-1}  \leq  v_s,\\
 x_1 + \cdots + x_{m-1} \leq  1,\\
x_1,...,x_{m-1} \geq  0.
\end{cases}
\end{equation}
Conversely, if $\alpha'_1,...,\alpha'_{m-1}$ is a feasible solution of (\ref{ELP2}), then $\alpha'_1,...,\alpha'_{m-1}, \beta'_1,...,\beta'_s$ is a solution of (\ref{ELP1}), where $\beta'_j = v_j - [(u_{1j}- u_{mj})\alpha'_1 + \cdots + (u_{(m-1)j }- u_{mj})\alpha'_{m-1} ]$, and the numbers $\alpha'_1,...,\alpha'_{m-1}, \alpha'_m, \beta'_1,...,\beta'_s$ satisfy  the relation (\ref{ELP0}), where  $\alpha'_m = 1- (\alpha'_1+\cdots + \alpha'_{m-1})$. 

Let $\Pcal\subset \Rset_+^{m-1}$ be the set of  solutions of (\ref{ELP2}). By virtue of (\ref{ELP0}), $\Pcal$ is a non-empty pointed polyhedron. Let $\albf^* = (\alpha^*_1,...,\alpha^*_{m-1})$ be a vertex of $\Pcal$.  By  a result of Hoffman and Kruskal \cite{HK} (see also, e.g., \cite[Theorem 8.4]{Sch}), one may assume that $\albf^*$ is the unique solution of a system of $m-1$ equations $A\xbf = \bbf$ for some subsystem $A\xbf \leq \bbf$ of the system of constrains (\ref{ELP2}). Let $N$ be the determinant of $A$. Then, by Cramer's rule, $\alpha^*_i = \alpha_i/N$, where $\alpha_i \in \Nset$.
 Since each entry of $A$ is either $u_{ij} - u_{mj}$, or $1$ or $0$, its absolute value is at most $\delta$. Hence, by Hadamard's inequality we get
 $$N\leq (m-1)^{(m-1)/2} \delta^{m-1} \leq s^{s/2}\delta^s.$$
 Then 
 $$ \begin{array}{ll}
\alpha^*_m  &= 1 - (\alpha^*_1+\cdots + \alpha^*_{m-1} ) = \alpha_m/N,\\
 \beta^*_j & = v_j - [(u_{1j}- u_{mj})\alpha^*_1 + \cdots + (u_{(m-1)j} - u_{mj})\alpha^*_{m-1} ]= \beta_j/N , \ (j=1,...,s),
\end{array}$$
for some non-negative integers $\alpha_m, \beta_1,...,\beta_s$. By the above argument, we get
$$\vbf =  \sum_{i=1}^m \alpha^*_i \ubf_i +  \sum_{j=1}^s \beta^*_j \ebf_j,$$
or equivalently,
$$N\vbf =  \sum_{i=1}^m \alpha_i \ubf_i +  \sum_{j=1}^s \beta_j \ebf_j.$$
\end{proof}

We can now state and prove the main result of this section.

\begin{thm} \label{Lead} Let $I$ be a monomial ideal. Let $q$ be the number of minimal generators from $G(I)$ of degree bigger than $\delta(I)$. Then:
\begin{itemize}
\item[(i)]  $p(I) = \delta(I)$. That means there is a non-negative integer $q_0(I)$ such that $d(I^n) = \delta(I) n + q_0(I)$ for all $n\gg 0.$
\item[(ii)] $r(I) \leq  q s^{s/2}\delta(I)^s - q$.
\item[(iii)] $\epsilon(I;n) \leq  q [ s^{s/2}\delta(I)^s - 1] [d(I)- \delta(I)]$ for all $n\geq 1$.
\end{itemize} 
\end{thm}

\begin{proof}   Set $\delta = \delta(I)$.
Assume that $I=(\Xbf^{\vbf_i}|\ i=1,...,p, .., t)$,  such that $V(I) =\{\vbf_1,...,\vbf_p\}$, $|\vbf_i| \leq \delta$ for all $i\leq t-q$ and $|\vbf_i| > \delta$ for all $i\geq t-q+1$, where $q\leq t-p$.
Let 
$$ \Jcal  := (\Xbf^{\vbf_1},..., \Xbf^{\vbf_p}) \subseteq 
J :=  (\Xbf^{\vbf_1},..., \Xbf^{\vbf_{t-q}}). $$

For each $j=t-q+1, ..., t$, $\vbf_j \in \conv(V(I)) + \Rset_+^s$. By Lemma \ref{LP}, there is a positive integer $N_j \leq s^{s/2}\delta^s$ and non-negative integers $\alpha_{j,i}, \gamma_{j,k}$ such that $\sum_{i=1}^p \alpha_{j,i} = N_j$ and
$$N_j \vbf_j = \sum_{i=1}^p \alpha_{j,i}\vbf_i + \sum \gamma_{j,k}\ebf_k.$$
Then $\Xbf^{N_j\vbf_j} \in \Jcal^{N_j}$. If $q=0$, let $N=0$. Otherwise, 
let 
\begin{equation}\label{ELead1} N = (N_{t-q+1}-1) + \cdots + (N_t - 1) \leq q (s^{s/2}  \delta^s- 1).
\end{equation}
Then 
\begin{equation}\label{ELead}
I^{n+1} =J  I^n   \ \text{ for all} \ n \geq N.
\end{equation}
Indeed, assume by contrary  that there is $m \in I^{n+1}\setminus J  I^n$. Then $m=  \prod_{j=t-q+1}^t\Xbf^{n_j\vbf_j}$ for some non-negative $n_j$ with $\sum n_j = n+1$. Assume $n_k \ge N_k$ for some $k$. Then 
$$\Xbf^{n_k\vbf_k} = \Xbf^{N_k\vbf_k} \Xbf^{(n_k-N_k)\vbf_k} \in \Jcal^{N_k}I^{n_k-N_k} \subseteq J^{N_k}I^{n_k-N_k} .$$
 Hence $m \in J^{N_k} I^{n+1 -N_k} \subseteq J I^n$, a contradiction. So all $n_j\leq N_j -1$ and $n+1 \leq N$, which contradicts to the condition $n\geq N$. Hence (\ref{ELead}) holds true. 
Now for all $n\geq N$ we have 
$$d(I^n) \leq d(J ^{n-N}) + d(I^N) \leq \delta  (n-N) + N d(I) = \delta  n + N(d(I)  - \delta ).$$
On the other hand, by Lemma \ref{MG1}, $d(I^n) \geq \delta n$. From (\ref{EKo}), it follows that   $p(I) = \delta $, and  $J$ is exactly the reduction of $I$ in the Kodiyalam's construction.  By virtue of Lemma \ref{CHTK}, $d(I^n) = \delta n + q_0(I)$ for some $q_0(I)$ and all $n\gg 0$. This shows (i). The statement (ii) follows from (\ref{ELead})  and (\ref{ELead1}), and (iii) follows from (ii) and Lemma \ref{G1}(ii).
\end{proof}

In \cite{H3}, we can give a number $n_0$ in terms of $s$ and $\delta(I)$ such that $d(\Inba) = \delta(I) n + c(I)$ for some $c(I)\geq 0$ and all $n\geq n_0$. Unfortunately we cannot solve the similar problem for $\gstab(I)$.  Lemma \ref{G4} already shows that in the worst case,  $\gstab(I)$ must depend on $\delta(I)$. The following example not only shows that the bound on the reduction number in the above theorem is close to be optimal, but also shows that in the worst case, the constant value of $\epsilon(I; n)$ as well as the index of stability $\gstab(I)$ are at least exponential functions of $s$, namely  $\epsilon(I; n), \gstab(I) >  (\delta(I)-1)^{s-3}$.

\begin{exm} \label{ExRed}
Given $s\geq 3$ and $\delta \geq s + 1$. Set
$$\begin{array}{ll}
\vbf_1 &= (\delta-1, 1, 0, 0, \cdots, 0, 0),\\
\vbf_2 & = (0, \delta-1, 1, 0, \cdots , 0, 0),\\
& \cdots\\
\vbf_{s-1} & = (0, 0, 0, 0, \cdots, \delta-1, 1),\\
\vbf_s &= (1,0,0,0,\cdots,0, \delta -1),
\end{array}$$
and $$\vbf = (\delta - s + 1,1,...,1).$$
Let $a$ be a positive integer. Consider the following  ideal in $s+1$ variables:
$$I = (M_1:=\Xbf^{\vbf_1}, ..., M_s:= \Xbf^{\vbf_s}, M:= \Xbf^{\vbf}Y^a ) \subset K[\Xbf, Y].$$
Then $V(I) = \{ (\vbf_1,0), ...,(\vbf_s,0)\}$ and the reduction $J$ of $I$ in the Kodiyalam's construction is $J= (\Xbf^{\vbf_1}, ..., \Xbf^{\vbf_s})$. Hence $\delta(I) = \delta$. 
\end{exm}

\begin{prop} \label{Red} Let $I$ be the monomial ideal given in Example \ref{ExRed}. Then:
\begin{itemize}
\item[(i)]  $(\delta - 1)^{s-2} < r(I) \leq (\delta -1)^s $;
\item[(ii)] $\epsilon(I; n) = an$ for all $n\leq r (I) -1$ and $\epsilon(I; n) = r(I) a > (\delta - 1)^{s-2} $ for all $n\geq r(I) $. In particular $\gstab(I) = r(I)  > (\delta - 1)^{s-2}$. 
\end{itemize} 
\end{prop}

\begin{proof} 

{\it Step 1:} To show $V(I) =  \{ (\vbf_1,0), ...,(\vbf_s,0)\}$.

 First, we find $\alpha_1,...,\alpha_s$ such that
\begin{equation}\label{EExRed1}
\alpha_1\vbf_1 + \cdots + \alpha_s \vbf_s = \vbf.
\end{equation}
Of course $(\alpha_1,...,\alpha_s)$ is a solution of the following system of equations
$$\begin{cases} \begin{array}{ccccccccccc}
(\delta-1) x_1 &  &   &  &   & & & +& x_s & =  &\delta - s +1,\\
x_1 &+ & (\delta - 1) x_2 &   &  & &  &   &    & =& 1,\\
 &  & x_2 & + & (\delta - 1) x_3 &   &   & &   & = &1, \\
 &   &  &   &   &   \cdots   &  &  &    &  &  \\
 & &  &   &  &    & x_{s-1} & +  & (\delta-1)  x_s  & = & 1.
\end{array}
\end{cases}$$
Using the Laplace expansion along the first row we can calculate the determinant of this system:
$$D = \begin{vmatrix}
\delta -1 & 0 & 0 & \cdots & 0 & 1\\
1 & \delta - 1& 0 & \cdots & 0 & 0\\
0 & 1 & \delta -1 & \cdots & 0 & 0\\
\vdots  & \vdots  & \vdots  & \ddots & \vdots & \vdots \\
0  & 0 & 0  & \cdots & 1  & \delta -1
\end{vmatrix} =  (\delta-1)^s + (-1)^{s-1} > 0.$$
So, the system (\ref{EExRed1}) has the unique solution. 

We show that all $\alpha_i > 0$. In order to compute the solution of (\ref{EExRed1}), let us first compute the following determinant of size $t\geq 2$:
$$B_{t,c}  = \begin{vmatrix}
c & \delta - 1& 0 & \cdots & 0 \\
1& 1 & \delta -1 & \cdots & 0 \\
\vdots  & \vdots  & \vdots  & \ddots & \vdots  \\
1  & 0 & 0  & \cdots & \delta - 1\\
1  & 0 & 0  & \cdots & 1  
\end{vmatrix} ,$$
where $c$ is any number. Using the Laplace expansion along the last row,  we get the recursion relation $B_{t,c}= (-1)^{t-1}(\delta - 1)^{t-1} + B_{(t-1),c}$. Hence
$$\begin{array}{ll}
B_{t,c} & = (-1)^{t-1}(\delta - 1)^{t-1} + \cdots + (-1)^2 (\delta - 1)^2 +  \begin{vmatrix}
c & \delta - 1\\ 1& 1\end{vmatrix} \\
&  =(-1)^{t-1} \frac{\delta - 1}{\delta} [ (\delta - 1)^{t-1} + (-1)^t] + c.
\end{array}$$
Denote by $D_i$ the determinant obtained from $D$ by replacing the $i$-th column by $(\delta -s + 1, 1,\cdots, 1)^T$. Then, using the Laplace expansion along the first row,  we now get 
$$\begin{array}{ll}
D_1 &= \begin{vmatrix}
\delta -s+1 & 0 & 0 & \cdots & 0 & 1\\
1 & \delta - 1& 0 & \cdots & 0 & 0\\
1 & 1 & \delta -1 & \cdots & 0 & 0\\
\vdots  & \vdots  & \vdots  & \ddots & \vdots & \vdots \\
1  & 0 & 0  & \cdots & 1  & \delta -1
\end{vmatrix} \\
& = (\delta-s+1) (\delta -1)^{s-1} + (-1)^{s-1}B_{(s-1),1} \\
& = (\delta-s+1) (\delta -1)^{s-1} + (-1)^{s-1}\{ (-1)^{s-2} \frac{\delta-1}{\delta}[ (\delta - 1)^{s-2} + (-1)^{s-1}] + 1\}.\\
& =  (\delta-s+1) (\delta -1)^{s-1}  - \frac{\delta-1}{\delta} [(\delta - 1)^{s-2} + (-1)^{s-1}] +  (-1)^{s-1}.
\end{array}$$
Clearly $D_1>0$, so $\alpha_1 = D_1/D >0$. 

For $1 < i < s$, using the Laplace expansion along either last or first columns several times,  we obtain:
$$\begin{array}{ll}
 D_i & = \begin{vmatrix}
 &   &  &  & \text{i-th} &   &   &  \\
\delta -1 & 0 & 0 & \cdots & \delta-s+1& \cdots & 0 & 1\\
1 & \delta - 1& 0 & \cdots & 1 & \cdots & 0 & 0\\
0 & 1 & \delta -1 & \cdots &1 & \cdots & 0 & 0\\
\vdots  & \vdots  & \vdots  & \ddots & \vdots & \ddots & \vdots & \vdots \\
0  & 0 & 0  & \cdots & 1 & \cdots & 1  & \delta -1
\end{vmatrix} \\
&=  (-1)^{s-1}  \begin{vmatrix}
 &   &  &  & \text{i-th} &   &     \\
1 & \delta - 1& 0 & \cdots & 1 & \cdots & 0 \\
0 & 1 & \delta -1 & \cdots &1 & \cdots & 0 \\
\vdots  & \vdots  & \vdots  & \ddots & \vdots & \ddots & \vdots \\
0  & 0 & 0  & \cdots & 1 & \cdots & 1  
\end{vmatrix} +
\end{array}$$
$$\begin{array}{ll}
& \hskip 1cm +  (\delta -1)
\begin{vmatrix}
 &   &  &  & \text{i-th} &   &    \\
\delta -1 & 0 & 0 & \cdots & \delta-s+1& \cdots & 0 \\
1 & \delta - 1& 0 & \cdots & 1 & \cdots & 0 \\
0 & 1 & \delta -1 & \cdots &1 & \cdots & 0 \\
\vdots  & \vdots  & \vdots  & \ddots & \vdots & \ddots & \vdots \\
0  & 0 & 0  & \cdots & 1 & \cdots   & \delta -1
\end{vmatrix} \\
& = (-1)^{s-1}B_{(s-i), 1} + (\delta-1)^{s-i} \begin{vmatrix}
 &   &  &  &  & \text{i-th}   \\
\delta -1 & 0 & 0 & \cdots & 0 & \delta-s+1 \\
1 & \delta - 1& 0 & \cdots &0 & 1  \\
0 & 1 & \delta -1 & \cdots& 0 &1 \\
\vdots  & \vdots  & \vdots  & \ddots& \vdots & \vdots \\
0  & 0 & 0  & \cdots & 1 & 1
\end{vmatrix} \\
& = (-1)^{s-1}B_{(s-i), 1} +(\delta-1)^{s-i} (-1)^{i-1} B_{i, \delta -s+1} .
\end{array}$$
To get the last equality we have moved the last column of the determinant in the preceding line  to the first position. So,
$$\begin{array}{ll}
 D_i & = (-1)^i  \frac{\delta - 1}{\delta} [ (\delta - 1)^{s-i-1} + (-1)^{s-i}] + (-1)^{s-1} + \\
& \hskip1cm  \frac{\delta - 1}{\delta} [(\delta -1)^{s-1} + 
 (-1)^{i}(\delta -1)^{s-i}] + (-1)^{i-1}  (\delta - 1)^{s-i} (\delta - s +1).
\end{array}$$
It is clear that $D_i> 0$ for $i\geq 3$. When $i= 2$, since
$$\frac{\delta - 1}{\delta} [(\delta -1)^{s-1} + 
 (\delta -1)^{s-2}] -   (\delta - 1)^{s-2} (\delta - s +1) = (\delta -1)^{s-2}(s-2) >0,$$
 also $D_2 > 0$. Hence $\alpha_i > 0$ for all $1< i < s$.
 Finally, notice that $D_s = (-1)^{s-1}B_{s, \delta - s+1}$ (moving  the last column of $D_s$ to the first position). Hence
 \begin{eqnarray} 
 \nonumber D_s & = & \frac{\delta - 1}{\delta} [ (\delta - 1)^{s-1} + (-1)^s] + (-1)^{s-1}(\delta - s+1) \\
\label{EExRed2} & = &
 \frac{(\delta-1)^s + (-1)^{s-1} (\delta^2 - \delta s + 1) }{\delta } > 0.
 \end{eqnarray} 
Summing up, we have showed that all $\alpha_i > 0$. Since $|\vbf_1| = \cdots = |\vbf_s| = |\vbf| = \delta$, it follows that $\alpha_1 + \cdots + \alpha_s =1$ and $D_1+\cdots + D_s = D$. Hence $\vbf \in \conv(\vbf_1,...,\vbf_s)$ and $(\vbf,a) \in \conv((\vbf_1,0),...,(\vbf_s,0)) + \Rset_+^{s+1}$, which yields  $V(I) = \{ (\vbf_1,0),...,(\vbf_s,0)\}$.
\vskip0.3cm

\noindent {\it Step 2}: To show (i).

Since $I = (J, M)$, by Remark \ref{EG2}, 
\begin{equation}\label{EExRed3}
r(I) = \min \{t > 0|\ M^t \in J^t \} - 1.
\end{equation}
As it is shown above 
$$\frac{D_1}{D}\vbf_1 + \cdots + \frac{D_s}{D}\vbf_s = \vbf,$$
and $D_1+\cdots  D_s = D$. Hence
$$M^D = Y^{Da}( M_1^{D_1} \cdots M_s^{D_s} ) \in J^D.$$
Therefore $r(I) <  D = (\delta -1)^s + (-1)^{s+1}$, whence $r(I) \leq (\delta -1)^s$.

On the other hand, by virtue of (\ref{EExRed2}),  the greatest common divisor of $D$ and $D_s$ is at most
$$ \begin{array}{l}
\gcd((\delta-1)^s + (-1)^{s-1} (\delta^2-\delta s+1), (\delta-1)^s  + (-1)^{s+1}) \\
\hskip0.5cm = \gcd((-1)^{s-1} (\delta^2-\delta s+1) - (-1)^{s+1}  , (\delta-1)^s  + (-1)^{s+1})  \\
\hskip0.5cm \leq \delta^2 - 3 \delta + 2 = (\delta - 1)(\delta -2).
\end{array}$$
Hence $\alpha_s = \frac{n_1}{n_2}$ with $\gcd(n_1,n_2) = 1$, and
$$ \begin{array}{ll} n_2 & = \frac{D}{\gcd(D_s,D)}\geq  \frac{(\delta-1)^s  + (-1)^{s+1}}{(\delta - 1)(\delta - 2)} \geq  \frac{(\delta-1)^s  - 1}{(\delta - 1)(\delta - 2)} \\
& = \frac{(\delta-1)^{s-1}}{\delta -2 } - \frac{1}{(\delta -1)(\delta -2)} = (\delta -1)^{s-2}  + \frac{(\delta-1)^{s-2}}{\delta -2 } - \frac{1}{(\delta -1)(\delta -2)} \\
& \geq  (\delta -1)^{s-2}  + \frac{\delta - 1}{\delta -2 } - \frac{1}{(\delta -1)(\delta -2)} = 
 (\delta - 1)^{s-2} + 1+  \frac{1}{\delta - 1 }.
 \end{array}$$
This implies $n_2 \geq (\delta - 1)^{s-2} + 2$. 

Now let $t>0$ be a number such that $M^t \in J^t$. Then there are  $t_1,...,t_s \in \Nset$ with $t_1+\cdots + t_s = t$ and $\ubf \in \Nset^s$ such that
$t\vbf = t_1\vbf_1 + \cdots + t_s\vbf_s + \ubf$. Since $|\vbf_1| = \cdots = |\vbf_s| = |\vbf| $, it follows that $\ubf = 0$. Then
$$\vbf= \frac{t_1}{t}\vbf_1 + \cdots + \frac{t_s}{t}\vbf_s.$$
Since the system (\ref{EExRed1}) has exactly one solution, we must have $\frac{t_s}{t} = \alpha_s = \frac{n_1}{n_2}$. By the choice of $n_1, n_2$, this implies $t\geq n_2$. Hence, by (\ref{EExRed3}), $r(I) \geq  n_2-1 >  (\delta - 1)^{s-2}$, which yields (i).
\vskip0.3cm

\noindent {\it Step 3}: To show (ii). 

Let  $1\leq t\leq r(I)$ and $n\geq 0$. Assume that $M^tM_1^n \in \mfr I^{t+n}$. Then one can find $t', n_1,...,n_s \in \Nset$ and $m' \neq 1$ a monomial in $K[\Xbf, Y]$ such that $t'+ n_1+ \cdots + n_s = t+n$ and 
$$M^t m_1^n = m' M^{t'} M_1^{n_1}\cdots M_s^{n_s}. $$  Looking at the exponents of $Y$ in both sides, we get $t'\leq t$. 
If $t'=t$, then  from $\deg(M_1) = \cdots = \deg(M_s)$ we must have $m'= 1$, a contradiction. Hence $t'< t$. Dividing both sides by $M^{t'}$ we can assume that $t'= 0$, i.e. 
$M^t m_1^n = m' M_1^{n_1}\cdots M_s^{n_s}$, where $n_1+\cdots + n_s = n+ t$. Then $t\vbf + n\vbf_1 = \ubf + n_1\vbf_1+ \cdots + n_s\vbf_s$ for some $\ubf \in \Nset^s$. Since $|\vbf| = |\vbf_1| = \cdots = |\vbf_s|$, we must have $\ubf =0$. Hence
$$\vbf = \frac{n_1-n}{t}\vbf_1 + \frac{n_2}{t}\vbf_2 + \cdots + \frac{n_s}{t}\vbf_s.$$
This means $( \frac{n_1-n}{t}, \frac{n_2}{t}, ...,\frac{n_s}{t})$ is a solution of (\ref{EExRed1}). Hence $\frac{n_1-n}{t} > 0$, or equivalently $n_1 > n$, and
$t\vbf = ( n_1 -n)\vbf_1+ \cdots + n_s\vbf_s$. This implies $M^t \in J^t$. By virtue of (\ref{EExRed3}) this  contradicts to the assumption $t\leq r(I)$. Hence $M^tM_1^n \not\in \mfr I^{t+n}$, that means $M^tM_1^n$ is a minimal generator of $I^{t+n}$, which yields $d(I^{t+n}) \geq \deg(M^t M_1^n) = \delta n + t a$, or equivalently $\epsilon(I; t+n) \geq ta$ for all $1\leq t\leq r$ and $n\geq 0$. Combining with the bound in Lemma \ref{G1}(ii) this gives the statement (ii) of the proposition.
\end{proof}

Our study of   degree excess functions gives interesting consequences on the Castelnuovo-Mumford regularity. For a finitely generated graded $R$-module $E$, the {\it Castelnuovo-Mumford regularity} of   $E$ is defined by
$$ \reg(E)= \max\{t|\ \text{there is}\ i \ \text{such that}\  H_{\mfr}^i (E)_{t-i} \ne 0\}, $$
where $H_{\mfr}^i (E)$ is the local cohomology module with the support  $\mfr$.  

In general, the Castelnuovo-Mumford regularity $\reg(I)$ of a homogeneous ideal of $R$ could be much larger than $d(I)$. However, it was independently  proved in \cite{CHT} and \cite{Ko} that $reg(I^n) = p(I)n + e(I)$ for some non-negative integer $e(I)$ and all $n\gg 0$. As an immediate consequence of Theorem \ref{Lead}, we obtain:

\begin{cor} \label{Lead3} Let $I$ be a monomial ideal. Then there is a non-negative integer $e(I)$ such that $\reg(I^n) = \delta(I) n +e(I)$ for all $n\gg 0$.
\end{cor}

Even in the case of monomial ideals, it is of great interest to give a bound on the index of stability of $\reg(I^n)$ defined as
 \begin{equation}\label{regstD}
 \regst(I) := \min\{t| \reg(I^n) = \delta(I) n +e(I) \ \text{for all}\ n\geq t\}.
\end{equation} 
 However, until now very small progress is achieved. The only existing bound is established for monomial ideals of dimension zero, see \cite[Theorem 3.1]{Ber}. On the other hand, a similar problem for the Castelnuovo-Mumford regularity $\reg(\Inba)$ of integral closures $\Inba$ of $I^n$ is solved in a recent paper \cite{H2}. In \cite{H3}
 we  show that $\delta(I)n \leq \reg(\Inba) \leq \delta(I) n+ \dim R/I$. In particular, $\reg(\Inba) = \delta(I) n + e'(I)$ for some $0\leq e'(I)\leq \dim R/I$ for all $n\gg 0$. Unfortunately we cannot use the technique presented in this paper to give a bound on $e(I)$ of Corollary \ref{Lead3} and on $\regst(I)$. However, our study on  degree excess functions leads to the following somewhat unexpected result, which says that in the worst case, the invariants $e(I)$ and $\regst(I)$ should be at least exponential functions of the number of variables.
 
 \begin{thm} \label{Reg}  Let $I$ be a homogeneous ideal of a polynomial ring $K[X_1,...,X_m]$ of $m\geq 4$ variables. Assume that $d(I)\geq 3$. Then, in the worst case we must have
 \begin{itemize}
\item[(i)] $e (I) > (d(I)-2)^{m-3}$,  and
\item[(ii)] $\regst(I) > \frac{1}{m-1}(d(I) - 2)^{m-4}$.
\end{itemize} 
\end{thm}

\begin{proof} Consider the ideal $I$ in Example \ref{ExRed} with $a=1$, $s=m-1$, $\delta \geq m$  and $Y= X_m$.
For short, let $t:= \regst(I)$ and $e:= e(I)$.  Then $\reg(I^t) = \delta t + e$. The greatest common divisor of minimal  generators of $I^t$ is $F= X_1^{(\delta-1) t} \cdots X_s^{(\delta - 1)  t} X_m^t$. By \cite[Theorem 3.1(a)]{BH} (see  also \cite[Theorem 3.1]{HTr}), $\reg(I^t) < \deg(F) = s (\delta -1) t + t$. Hence 
$$e < (s-1) ( \delta - 1) t + t < s(\delta -1)t.$$
On the other hand, since $\delta n + \epsilon (I;n) = d(I^n) \leq \reg(I^n)$ for all $n$, taking $n\gg 0$ we get $\epsilon(I; n) \leq e$. By Proposition \ref{Red}, it implies
$$e\geq  r(I) > ( \delta - 1)^{s-2} = (d(I)-2)^{m-3},$$
whence
$ s( \delta - 1)  t >  e  > (\delta - 1)^{s-2}$. Hence
$$ t > \frac{1}{s}(\delta - 1)^{s-3} = \frac{1}{m -1}(d(I) - 2)^{m-4}.$$
\end{proof}

\noindent {\bf Acknowledgement}.  I would like to thank the referee for careful reading and useful suggestions. In particular, the proof of Lemma \ref{G5} is due to him/her.

 This work is partially supported by NAFOSTED (Vietnam) under the grant number 101.04-2018.307 and the Program for Research Activities of Senior Researchers of VAST under the grant number  NVCC01.11/21-21.

\end{document}